\pdfoutput=1

\documentclass[12pt]{article}
\usepackage[a4paper]{geometry}
\usepackage[utf8]{inputenc}
\usepackage{amsmath, amsthm, amssymb, amsfonts}
\usepackage{parskip}
\usepackage{mathrsfs}
\usepackage{graphicx}
\usepackage{xtab}
\usepackage{xspace}
\usepackage[shortlabels,inline]{enumitem}
\usepackage{lipsum}
\usepackage{authblk}

\usepackage{hyperref}
\hypersetup{
    colorlinks,
    citecolor=black,
    filecolor=black,
    linkcolor=black,
    urlcolor=black
}

\newcommand\blfootnote[1]{%
  \begingroup
  \renewcommand\thefootnote{}\footnote{#1}%
  \addtocounter{footnote}{-1}%
  \endgroup
}

\newcommand{\ideal}[1]{\langle #1\rangle}

\DeclareMathOperator{\ppdim}{ppdim}

\DeclareMathOperator{\size}{dist_p}
\DeclareMathOperator{\dep}{depth}

\DeclareMathOperator{\ind}{ind}
\DeclareMathOperator{\Char}{char}

\usepackage{parskip}
\usepackage{graphicx}
\usepackage[nodayofweek]{datetime}
\usepackage{tikz}
\usepackage{tikz-cd}
\usepackage{setspace}

\newdateformat{monthyear}{\monthname[\THEMONTH], \THEYEAR}

\theoremstyle{definition}
\newtheorem{definition}{Definition}[section]

\theoremstyle{plain}
\newtheorem{theorem}[definition]{Theorem}
\newtheorem*{theorem*}{Theorem}
\newtheorem{corollary}[definition]{Corollary}
\newtheorem{proposition}[definition]{Proposition}
\newtheorem{lemma}[definition]{Lemma}

\theoremstyle{remark}
\newtheorem{remark}{Remark}
\newtheorem*{remark*}{Remark}
\newtheorem{example}{Example}
\newtheorem*{example*}{Example}
\newtheorem{question}{Question}

\begin{document}
\title{\textbf{Permutation Dimensions of Prime Cyclic Groups}}
\author[1]{Jack Walsh}
\affil[1]{Simion Stoilow Institute of Mathematics of the Romanian Academy}
\date{}
\maketitle{}
\begin{abstract}
\noindent
Based on recent successes concerning permutation resolutions of representations by Balmer and Gallauer we define a new invariant of finite groups: the $p$-permu\-ta\-tion dimension. We define this analogously to the global dimension of a ring by replacing projective resolutions of ring modules with resolutions by $p$-permutation modules of modules over the group ring. We compute this invariant for cyclic groups of prime order.
\end{abstract}
\blfootnote{This work was was supported by the project “Group schemes, root systems, and related representations” founded by the European Union - NextGenerationEU through Romania’s National Recovery and Resilience Plan (PNRR) call no. PNRR-III-C9-2023-I8, Project CF159/31.07.2023, and coordinated by the Ministry of Research, Innovation and Digitalization (MCID) of Romania.}
\section{Introduction}
\label{sec:introduction}
\emph{Throughout we fix a field~$k$ of characteristic $p>0$.
All modules are finitely generated.}
\vspace*{1em}

Let $G$ be a finite group.
Recall that a $k G$-module is a \emph{permutation module} if it admits a $G$-invariant $k$-basis.
More generally, it is a \emph{$p$-permutation module} if it admits a $k$-basis which is invariant under the action of some $p$-Sylow subgroup.
Equivalently, a $p$-permutation module is nothing but a direct summand of a permutation module.

Since a permutation module can be thought of as the equivariant analogue of a free module (i.e. it is free ``up to the action of G'') this last characterization suggests that $p$-permutation modules are equivariant analogues of projective modules
Those, of course, are widely used to `resolve' arbitrary modules and play an important role even in equivariant mathematics.
However, there is a caveat: If $p$ divides the order of~$G$, no non-projective $k G$-module admits a finite projective resolution.
In other words, the projective dimension of a $k G$-module is either~$0$ or~$\infty$: a rather coarse invariant.

\medskip
It has recently been shown that their equivariant analogues fare better~\cite{MR4541331}: Every $k G$-module admits a finite $p$-permutation resolution.
That is, every $k G$-module is built out of $p$-permutation modules in finitely many steps, and it is natural to ask how many steps are necessary.
We thereby hope to obtain a finer measure of the complexity of a $k G$-module.
In particular, taking the supremum over all $k G$-modules we define the \emph{$p$-permutation dimension} of~$G$, and we might hope to learn something about~$G$ by computing this invariant.

It should be said that the proof in~\cite{MR4541331} is non-constructive, and even if one could extract from it $p$-permutation resolutions, there is no reason why these would be of minimal length.
In fact, both constructing such finite resolutions and establishing their minimality seem to be difficult problems in general.
In this article we take some first steps by computing the $p$-permutation dimension for prime cyclic groups.

\medskip
Recall that for the cyclic group~$C_p$ the indecomposable $k C_p$-modules $M_i$ are classified by their dimension $i\in\{1,\ldots,p\}$. Since $C_p$ is a $p$-group the $p$-permutation modules are exactly the permutation modules. Here the indecomposable permutation modules are exactly $M_1$ and $M_p$ (i.e. the trivial module $k$ and the free module $kC_p$).
We compute the permutation dimension of each $M_i$ and show that $\ppdim(\bigoplus_iM_i) = \max\{\ppdim(M_i)\}$ (Theorem~\ref{thrm:ppdim=size}). We thereby arrive at the main result (Corollary~\ref{cor:groupdim}):
\begin{theorem*}
Let $k$ be a field of characteristic $p$. Then $\ppdim_{k}(C_p)= p-2$.
\end{theorem*}

\section{Definition and Basic Properties}
\label{sec:def-basic}

Let $G$ be any finite group.
For basics on $p$-permutation modules we refer to~\cite[\S\,14]{MR1416328}.
The following is the main notion we will be concerned with in this paper.
\begin{definition}
Let $M$ be a $k G$-module.
\begin{enumerate}
\item A \emph{$p$-permutation resolution} of~$M$ is an exact complex
\[
0\to P_s\to P_{s-1}\to\cdots\to P_0\to M\to 0
\]
in which the $P_i$ are $p$-permutation modules.
We call~$s$ its \emph{length}.
\item The \emph{$p$-permutation dimension} of~$M$ is the minimal length of all $p$-permutation resolutions of $M$.
We denote this by $\ppdim(M)$.
\end{enumerate}
\end{definition}

\begin{remark}
By~\cite{MR4541331}, there are $p$-permutation resolutions of finite length and hence there is a shortest such.
In particular, $\ppdim(M)$ is well-defined.
\end{remark}

\begin{definition}
\label{defn:ppdim}
We define the \emph{$p$-permutation dimension} of~$G$ (over~$k$) as
\[
\ppdim_{k}(G):=\sup_M\ppdim(M)
\]
where $M$ ranges over all $k G$-modules.
This is a non-negative integer or~$\infty$.
\end{definition}

\begin{example}
\begin{enumerate}[(a)]
\item Assume $p$ does not divide the order of~$G$.
Obviously, every $k G$-module is $p$-permutation (even projective).
It follows that $\ppdim_{k}(G)=0$.
\item Let $p=2$ and $G=C_2$ the cyclic group of order~$2$.
Then every $k G$-module is permutation hence again $\ppdim_{k}(C_2)=0$.
Note that in contrast to this, the global dimension of $kC_2$ is infinite since the trivial module~$k$ is non-projective.
\item Let $p=3$ and $G=C_3$ the cyclic group of order~$3$. It is easy to see that $M_2$ admits a permutation resolution of length $1$. Since $M_1$ and $M_3$ are permutation modules, the maximal permutation dimension of an $\textit{indecomposable}$ $kC_3$-module is $1$. The main result will show that this implies that $\ppdim_{k}(C_3)=1$.
\end{enumerate}
\end{example}

Consider the following basic property of $p$-permutation dimensions.

\begin{proposition}
\label{prop:ppdim-sum-product}
Let $M$ and $M'$ be $k G$-modules.
Then we have:
\begin{align}
  \ppdim(M \oplus M') &\leq \max\{\ppdim(M), \ppdim(M')\}
  \label{eq:ppdim-sum}\\
  \ppdim(M \otimes_k M') &\leq \ppdim(M) + \ppdim(M')
  \label{eq:ppdim-product}                          
\end{align}
\end{proposition}
\begin{proof}
Choose a minimal $p$-permutation resolution of~$M$ and~$M'$, respectively.
Their degreewise sum is a $p$-permutation resolution of $M\oplus M'$.
Their tensor product is a $p$-permutation resolution of $M\otimes M'$.
The claim follows immediately.
\end{proof}
Note that when computing the $p$-permutation dimension of $kG$ it suffices to compute only the dimension of indecomposable modules by $(1)$.
\begin{question}
It is natural to ask whether~\eqref{eq:ppdim-sum} and~\eqref{eq:ppdim-product} are equalities.
We answer these questions for \emph{prime cyclic} groups, see Corollary~\ref{corr:cyclicsum}.
\end{question}

\section{Permutation Dimension}
\label{sec:cyclic-p-groups}
For the remainder of this article we assume $G$ is the cyclic group of order $\Char(k) = p$.

Letting $T$ denote the difference between a generator and the identity element we identify
\[
k G=k[T]/T^{p}.
\]
For any $k G$-module~$M$ we have a decreasing filtration by the powers of the radical
\[
M=T^0M\supseteq T^1M\supseteq\cdots\supseteq T^{p}M=0
\]
\begin{definition}
Take $m \in M$.

The \emph{depth of m}, denoted $\dep(m)$, is defined as the maximal $i$ such that $m\in T^iM$ (and $\dep(0)=\infty$).

The \emph{index of nilpotency of $m$ with respect to $T$} is defined as the minimal $i$ such that $T^im = 0$. We denote this by $\ind_T(m)$ and may refer to it simply as the index of $m$.

We also denote by $\ideal{m}\subseteq M$ the $k G$-submodule generated by $m\in M$.
\end{definition}

\begin{remark}
\label{rmk:structure-theorem}
Applying the structure theorem for principal ideal domains to $k[T]$ we may obtain set of generators $\{m_i\}$ of $M$ such that each $\ideal{m_i}$ is a nonzero direct summand of $M$. Moreover, this direct summand is isomorphic to $k[T]/T^{(m)}=:M_{\ind_T(m)}$ and thus such a basis induces a decomposition $M\cong \oplus_iM_{\ind_T(m_i)}$.

We note that the multiset of these indices is independent of the choice of such a basis. These are called the \emph{invariants} of~$M$. It will be important for us to determine when an arbitrary nonzero $m \in M$ generates a direct summand.
\end{remark} 

\begin{lemma}
\label{lemm:summandcriteria}Let $m\in M$ be nonzero and of index $\alpha > 0$. The following are equivalent:
\begin{enumerate}
    \item $m$ generates a direct summand of $M$,
    \item $\dep(T^im) = i$ for all $i < \alpha$,
    \item $\dep(T^{\alpha-1} m) = \alpha - 1$.
\end{enumerate}
\end{lemma}
\begin{proof}
We first show that $1$ implies $2$. Let $M = \langle m \rangle \oplus M'$ and assume for contradiction that for some $i<\alpha$ we have $T^im = T^jx$ with $j > i$ and $x = fm + m' \in M$ with $f \in k[T]$. Now
\begin{align*}
T^{\alpha - 1}m &= T^{\alpha-1 + (j-i)}(fm + m')\\
&= T^{\alpha-1 + (j-i)}m'
\end{align*}
since $(j-i) \geq 1$ and $T^\alpha m = 0$. Hence $T^{\alpha-1}m \in \langle m \rangle \cap M' = \{0\}$. But $T^{\alpha-1}m \neq 0$ by definition and we obtain a contradiction.

The fact that 2 implies 3 is clear. Hence we show that 3 implies 1.

We choose a decomposition $M=\oplus\ideal{m_i}$ where $\ind_T(m_i)=\alpha_i$ and write $m=\sum_if_im_i$ with $f_i\in k[T]$.
We consider $T^{\alpha-1}m = \sum_iT^{\alpha-1}f_im_i$. By assumption there exists $i_0$ such that $\dep(T^{\alpha-1}f_{i_0}m_{i_0}) = \alpha - 1$ and therefore $\dep(f_{i_0}m_{i_0}) = 0$. We see therefore that $f_{i_0} (0) \in k^\times$ and we claim this means that $f_{i_0}$ is invertible in $k[T]/T^{\alpha}$. Indeed, setting $f := f_{i_0}$ we can construct such an inverse $g \in k[T]/T^{\alpha}$ as follows.

Let $f = \sum_{i\geq0}\lambda_iT^i$. By above $\lambda_0 \neq 0$ so we may set $g_0 = 1/\lambda_0$ so that 
\begin{align*}
gf = 1 + \sum_{i\geq 1}\lambda_iT^i.
\end{align*}
We now set 
\begin{align*}
g_1 = g_0 - (\lambda_1/\lambda_0^2)T
\end{align*}
so that 
\begin{align*}
g_1f &= 1 + \sum_{i\geq1}(\lambda_i/\lambda_0 - \lambda_1\lambda_{i-1}/\lambda_0^2)T^i\\
&= 1 + \sum_{i\geq2}\lambda'_iT^i.
\end{align*}
Iterating this product we may produce a $g \in k[T]$ such that 
\begin{align*}
gf = 1 + \sum_{i\geq\alpha}\lambda_i'T^{\alpha}.
\end{align*}

Now define a morphism $\pi\colon M\to \ideal{m_{i_0}}\xrightarrow{g}\ideal{m}$ where the first arrow is the canonical projection and the second sends $m_{i_0}$ to $gm$.
We have $\pi(m)=g fm=m$, establishing the claim.
\end{proof}

We have a slightly stronger criterion for a element of a \emph{permutation module} to generate a direct sum.

\begin{lemma}
\label{lemm:permsplitting}
Let $P$ be a permutation module and let $m \in P$ be nonzero. Then $m$ generates a direct summand iff either 
\begin{itemize}
\item $\dep(Tm) = 1$
\item $\dep(m)=0$ and $\ind_T(m) = 1$.
\end{itemize}
\end{lemma}
\begin{proof}
The forwards implication follows immediately from Lemma~\ref{lemm:summandcriteria}.

For the reverse implication the second bullet point also follows immediately from Lemma~\ref{lemm:summandcriteria}. Note that this is exactly when $\langle m\rangle \cong k$. 

Thus assume $\ind_T(m) = \alpha > 1$. Since $P$ is a permutation module we can write $P = \bigoplus_r kC_p \bigoplus_sk$. Let $\{m_1, \dots , m_r, m_{r+1}, \dots m_{r+s}\}$ be a set of generators of the summands of $P$ with $\ind_T(m_i) = p$ for $1 \leq i \leq r$ and $\ind_T(m_i) = 1$ otherwise. Let $m = \sum_{i=1}^{r+s}f_i(T)m_i$ so that $Tm = \sum_{i=1}^r Tf_i(T)m_i$. By the assumption that $\dep(Tm) = 1$ there exists some $1\leq j \leq r$ such that $f_j$ has nonzero constant term. Since $m_j$ generates a summand of the form $kC_p$ we have $T^{p-1}m_j \neq 0$ and hence $T^{p-1}m \neq 0$. We claim now that $m$ must generate a direct summand: otherwise there would exist $l > p-1$ and $m' \in P$ such that $T^{p-1}m = T^lm'$ by Lemma~\ref{lemm:summandcriteria}. But $T^lM = 0$ by definition contradicting the fact that $T^{p-1}m \neq 0$.
\end{proof}

We will now define, for an arbitrary $kC_p$-module $M$, a function taking values in $\mathbb{N}$ which depends only on the dimensions of its irreducible submodules. The value of this function will turn out to be exactly equal to the permutation dimensions of $M$.
\begin{definition}
Let $1 \leq x \leq p$. We define the \emph{$p$-distance of an integer $x$}, denoted $\size(x)$, inductively as:
\begin{itemize}
\item $\size(x) = \min\{\size(p-x), \size(p-x+1)\} + 1$ for $2 \leq x \leq p-1$, and
\item $\size(1) = \size(p) = 0$.
\end{itemize}
\end{definition}
\begin{remark}
\label{rmrk:sizesgraph}
One should think of the $p$-distance of $x$ as the number of ``moves'' required to reach 1 from $x$, where each ``move'' consists of passing from $x$ to either $p-x$ or $p-x+1$. For $p>3$ the value $\size(x)$ can be visualised by considering the iteratively constructed diagram
\[
\begin{tikzcd}
1 \ar[r, no head] 
&p-1 \ar[r, no head] 
&2 \ar[r, no head] 
& p-2 \ar[r, no head]
& \dots \ar[r, no head]
&\frac{p-1}{2} \ar[r, no head] 
& \frac{p+1}{2}
\end{tikzcd}
\]
where the $p$-distance strictly increases by $1$ from left to right.  Note that the integer to the right of $x$ is alternately $p-x$ and $p-x+1$.
\end{remark}
\begin{definition}
Let $M$ be a $k G$-module.
We define $\size(M)$ as the maximal $\size(x)$ where~$x$ ranges through the invariants of~$M$.
\end{definition}
Again one should think of this as the ``distance'' of $M$ from a $p$-permutation module. More properly, the relation of this $p$-distance with the permutation dimension is given by the following results.

\begin{proposition}
\label{prop:mainprop}
Let $P$ be a permutation module such that
\begin{align*}
0 \rightarrow K \xrightarrow{g} P  \xrightarrow{f} M \rightarrow 0
\end{align*}
is a short exact sequence. Furthermore assume that $f$ (resp. $g$) does not induce an isomorphism between a direct summand of $P$ and one of $M$ (resp. $K$ and $P$). Then $\size(K) \geq \size(M) - 1$.
\end{proposition}

\begin{proof}
We will prove a slightly stronger result: for every invariant~$c$ of~$M$ there is an invariant~$c'$ of~$K$ such that $\size(c')\geq \size(c)-1$.

We decompose $M$ into direct summands with generators $m_1,\ldots, m_l $ in $M$ and set $x_i=\ind_T(m_i)$. We claim that we can lift this to $\tilde{m}_1,\ldots, \tilde{m}_l\in P$ along~$f$ such that each $\tilde{m}_i$ generates a direct summand of $P$. If $x_i >1$ then any $\tilde{m}_i$ in $f^{-1}(m_i)$ will suffice by Lemma~\ref{lemm:permsplitting}. Hence assume $x_i = 1$ and that $\tilde{m}_i \in f^{-1}(m_i)$ does not generate a direct summand of $P$. By Lemma~\ref{lemm:permsplitting} we can write $T\tilde{m}_i = T^{l}\tilde{m}$ for $l > 1$ and $\tilde{m} \in P$. Set $\tilde{m}'_i = \tilde{m}_i - T^{l-1}\tilde{m}$. We have $T\tilde{m}_i' = 0$ and $f(\tilde{m}_i') = f(\tilde{m}_i)$. Thus $\tilde{m}'_i$ gives the desired lift.

We now claim that each $\tilde{m}_i$ generates a direct summand of $P$ isomorphic to $M_p$. This is equivalent to showing that $\ind_T(\tilde{m}_i) = p$ for all $i$. As $P$ is a permutation module we have $\ind_T(\tilde{m}_i)\in\{1,p\}$. But if $\ind_T(\tilde{m}_i) = 1$ for some $i$ then $Tm_i = f(T\tilde{m}_i) = 0$. Hence $f$ induces an isomorphism from $\langle \tilde{m}_i \rangle$ to $\langle m_i \rangle$, contradicting the assumption. We claim also that $M$ does not have a direct summand of the form $M_p$ by the exact same argument.

Consider then the following commutative diagram: 
\[
\begin{tikzcd}
&
0\ar[d]
&
0\ar[d]
&
0\ar[d]
\\
0 
\ar[r]
&
\oplus_iM_{p-x_i}
\ar[r, "(T^{x_i})"]
\ar[d, "j"]
&
\oplus_iM_{p}
\ar[r]
\ar[d, "(\tilde{m}_i)"]
&
\oplus_iM_{x_i}
\ar[d, "(m_i)"]
\ar[r]
&
0
\\
0
\ar[r]
&
K
\ar[r, "g"]
\ar[d, "\pi"]
&
P
\ar[d]
\ar[r, "f"]
&
M
\ar[d]
\ar[r]
&
0
\\
0
\ar[r]
&
K'
\ar[r]
\ar[d]
&
P'
\ar[r]
\ar[d]
&
0
\\
&
0
&
0
\end{tikzcd}
\]
The top row is the short exact sequence
\[
\begin{tikzcd}
0 
\ar[r]
&
\oplus_iM_{p-x_i}
\ar[r, "(T^{x_i})"]
&
\oplus_iM_{p}
\ar[r]
&
\oplus_iM_{x_i}
\ar[r]
&
0
\end{tikzcd}
\]
where the right hand map is induced by the map $M_p \to M_{x_i}$ sending generator to generator for each $i$. The left hand map is the inclusion of the kernel of the right hand map into $\oplus_i M_p$.

The top middle map vertical map sends the generator of the $i$-th copy of $M_p$ to $\tilde{m}_i$ and similarly the top right vertical map sends the generator of $M_{x_i}$ to $m_i$.  By the first part of the proof the map from $\oplus_i M_p$ to $P$ sends generators of copies of $M_p$ to generators of copies of $M_p$. We should note that the choice of generator is not unique (indeed a generator can be multiplied by any element of $k^\times$ to obtain another one), but it is clear that one can make this choice so that the top right square commutes.

The map $j$ is defined such that the top left square commutes. It is well defined by injectivity of the other maps in this square, and by commutativity of the top right square. The modules $K'$ and $P'$ are such that the vertical sequences are exact, i.e. they are the cokernels of the top left and top middle vertical maps. Exactness of the bottom row follows from the snake lemma.

We note that the middle column is split exact and therefore that $P'$ (and hence $K'$) is a permutation module. Note that we may assume the invariants of $K'$ are all $1$. Indeed if $K'$ has a direct summand of the form $M_p$ then so do $K$ and $P$ by surjectivity of the bottom left and bottom middle maps and projectivity of $M_p$. But then $g$ induces an isomorphism between these two summands, contradicting the assumption.

For $i\in\{1,\ldots,l\}$ write $1_i\in M_{p-x_i}$ for the identity. Let $\dep(j(1_i))=y_i\geq 0$ in $K$ and choose $k_i\in K$ such that $T^{y_i}k_i=j(1_i)$. Note that since the invariants of $K'$ are all $1$, $\pi(Tk_i) = 0$ and hence is in the image of $j$. Thus $y_i \leq 1$ by injectivity of $j$.

If $k_i$ generates a direct summand we have obtained a direct summand of $K$ with $\ind_T(k_i) = \ind_T(1_i) + y_i \in \{p-x_i, p-x_i+1\}$.

Hence suppose that $k_i$ does not generate a direct summand of $K$. By Lemma~\ref{lemm:summandcriteria} there exists $k \in K$ and $l \geq \ind_T(k_i)$ such that $T^{\ind_T(k_i)-1}k_i=T^l k$. As above we have $x\in\oplus_iM_{p-x_i}$ such that $j(x) = Tk$ and hence $j(T^{\ind_T(k_i)-1-y_i}1_i) = j(T^{l-1}x) $. By injectivity we have $T^{\ind_T(k_i)-1-y_i}1_i = T^{l-1}x$. Here if either $y_i = 1$ or $l> \ind_T(k_i)$ we see that $\ind_T(k_i) -1-y_i < l-1$, contradicting the fact that $1_i$ generates a direct summand by Lemma~\ref{lemm:summandcriteria}. In this case we conclude that $k_i$ generates a direct summand.

Hence we consider the case $y_i =0$ and $l = \ind_T(k_i)$. Here we claim that instead $k$ generates a direct summand of $K$. If not, again by Lemma~\ref{lemm:summandcriteria} there exists $k' \in K$ and $l' \geq \ind_T(k)$ such that $T^{\ind_T(k)-1}k = T^{l'}k'$. In particular we have that $\ind_T(k) = l+1$ since $T^{l}k = T^{\ind_T(k_i)-1}k_i$ and again we have $x' \in\oplus_i M_{p-x_i}$ such that $j(x')=Tk'$. Unwinding this we see that $T^{\ind_T(k_i)-1}1_i=T^{l-1}x=T^{l'-1}x'$. Here we have $\ind_T(k_i)-1 <l'-1$ and we obtain the same contradiction as in the $y_i = 1$ case. Note that by assumption $\ind_T(k) = \ind_T(k_i) + 1 = p-x_i + 1$.

In all cases we obtain the fact that $K$ has a summand generated by an element $k_i$ with $\ind_T(k_i) \in \{p-x_i, p-x_i+1\}$. This gives the desired result.
\end{proof}


\begin{theorem}
\label{thrm:ppdim=size}
Let $M$ be a $k G$-module.
Then $\ppdim(M)=\size(M)$.
\end{theorem}
\begin{proof}
We first show that $\ppdim(M) \geq \size(M)$. To do this we induct on $s=\ppdim(M)$.
If $s=0$ then $M$ is a permutation module and $\size(M)=0$.
Now assume $s>0$.
Consider a minimal permutation resolution
\[
0\to P_s\xrightarrow{f_s}\cdots \to P_0\xrightarrow{f_0} M\to 0.
\]
By Remark 2.11 in \cite{BALMER2022108535} we may assume that this resolution is an indecomposable chain complex. In particular, letting $K_0=\ker(f_0)$, we have that the short exact sequence
\begin{align*}
0 \rightarrow K_0 \rightarrow P_0  \rightarrow M \rightarrow 0
\end{align*}
satisfies the assumptions of  Proposition~\ref{prop:mainprop}. Applying this we get $\size(K_0)\geq \size(M)-1$.
By induction, $s-1=\ppdim(K_0)\geq \size(K_0)\geq \size(M)-1$ so that $s\geq \size(M)$, as required.

We now show that $\ppdim(M)\leq \size(M)$.
We prove this by induction on~$\size(M)$.
If $\size(M)=0$ then $M$ is a permutation module and hence $\ppdim(M)=0$ as well.

Now assume $\size(M)=s+1>0$.
Fix a direct sum decomposition $M=\oplus_i M_{x_i}$, where $\ind_T(m_i)=x_i$.

If $x_i \in\{1,p\}$ we set $P(i):= M_{x_i}$ together with the map $P(i)\to M$ that sends~$1$ to~$m_i$.

Otherwise we may write $x_i=p+\varepsilon_i-x_i'$ where $\varepsilon_i\in\{0,1\}$, $p> x_i> 1$ and $\size(x_i')= \size(x_i)-1$. Define a map $M_{p}\to M_{x_i}$ by sending~$1$ to~$m_i$.
This is well-defined because $p> x_i$.
If $\varepsilon_i=1$ also define $M_{1}\to M_{x_i}$ by sending~$1$ to~$T^{x_i-1}m_i$.This is again well-defined because $x_i > 1$.
If $\varepsilon_i=0$ we set $P(i):= M_{p}$ and if $\varepsilon_i=1$  we set $P(i):=M_{p}\oplus M_{1}$.

In any case, we have a permutation module $P(i)$ together with a morphism $f(i)\colon P(i)\to M_{x_i}$ that surjects onto $\ideal{m_i}$.
If $x_i$ is a $p$-power then $\ker(f(i))=0$ and otherwise $\ker(f(i))\cong M_{x_i'}$ so that $\size(\ker(f(i)))= \size(x_i)-1$.
We then have a short exact sequence
\begin{equation}
\label{eq:ses}
0\to K:=\oplus_iM_{x_i'}\to P:=\oplus_i P(i)\to M\to 0
\end{equation}
where $P$ is a permutation module and $K$ has $p$-size~$s$.
By induction $\ppdim(K)= s$ so there exists a permutation resolution of~$K$ of length~$s$.
Splicing it together with~\eqref{eq:ses} gives a permutation resolution of~$M$ of length~$s+1$ so that $\ppdim(M)\leq s+1=\size(M)$.
\end{proof}

Theorem~\ref{thrm:ppdim=size} gives the permutation dimension of any $kG$-module $M$ in terms of the dimensions of its indecomposable summands. From this we obtain some immediate corollaries.

\begin{corollary}
\label{cor:groupdim}
Let $k$ be a field of characteristic $p$. Then $\ppdim_{k}(C_p)= p-2$.
\end{corollary}
\begin{proof}
From Theorem~\ref{thrm:ppdim=size} we see that $\ppdim_{k}(C_p) = \max\{\size(x)\}$ where $1 \leq x \leq p-1$. It is easy to see that this is exactly $\size(\frac{p+1}{2}) = p-2$ (see Remark~\ref{rmrk:sizesgraph}).
\end{proof}

We recall from Proposition~\ref{prop:ppdim-sum-product} that the permutation dimension of a direct sum is bounded above by the maximum of the dimensions of the summands. Theorem~\ref{thrm:ppdim=size} tells us that this is an equality when $G = C_p$.

\begin{corollary}
\label{corr:cyclicsum}
Let $M$ and $N$ be $kC_p$-modules. Then
\begin{align*}
\ppdim(M \oplus N) = \max\{\ppdim(M), \ppdim(N)\}.
\end{align*}
\end{corollary}
\begin{proof}
It is clear from the definition that $\size(M \oplus N) = \max \{\size(M), \size(N)\}$. This gives the desired result.
\end{proof}

On the other hand note that the corresponding inequality for tensor products cannot be an equality for $p \geq 3$. This follows immediately from the fact that $\ppdim_{k}(C_p)= p-2$ is non-zero and bounded.

\section*{Acknowledgements}
I would like to give thanks to my master's degree supervisor Martin Gallauer, both for providing such an interesting yet approachable problem for my thesis (on which this article is based) and for many insightful and enjoyable discussions throughout the project and beyond.
\bibliographystyle{vancouver}
\bibliography{bibliography.bib}

\begin{thebibliography}{1}
\expandafter\ifx\csname urlstyle\endcsname\relax
 \providecommand{\doi}[1]{doi:\discretionary{}{}{}#1}\else
 \providecommand{\doi}{doi:\discretionary{}{}{}\begingroup \urlstyle{rm}\Url}\fi

\bibitem{BALMER2022108535}
Balmer, P., Gallauer, M. (2022).
\newblock Three real artin-tate motives.
\newblock \emph{Advances in Mathematics}, 406: 108535.
\newblock \doi{https://doi.org/10.1016/j.aim.2022.108535}.

\bibitem{MR4541331}
Balmer, P., Gallauer, M. (2023).
\newblock Finite permutation resolutions.
\newblock \emph{Duke Math. J.}, 172(2): 201--229.
\newblock \doi{10.1215/00127094-2022-0041}.

\bibitem{MR1416328}
Karpilovsky, G. (1995).
\newblock \emph{Group representations. {V}ol. 4}, vol. 182 of \emph{North-Holland Mathematics Studies}.
\newblock North-Holland Publishing Co., Amsterdam.

\end{thebibliography}
\end{document}